\documentclass[11pt]{amsart}
\usepackage{amssymb, amsmath, amsfonts, amsthm, mathrsfs}

\DeclareMathOperator{\GCD}{GCD}
\DeclareMathOperator{\LCM}{LCM}

\newcommand{\FF}{\mathbb{F}}

\newcommand{\ZZ}{\mathbb{Z}}

\newcommand{\cA}{\mathcal{A}}

\newcommand{\cK}{\mathcal{K}}

\newcommand{\cP}{\mathcal{P}}

\newcommand{\cR}{\mathcal{R}}

\newcommand{\be}{\boldsymbol{e}}
\renewcommand{\bf}{\boldsymbol{f}}
\newcommand{\bg}{\boldsymbol{g}}
\newcommand{\bp}{\boldsymbol{p}}
\newcommand{\wcK}{\widehat{\cK}}

\theoremstyle{plain}
\newtheorem{theorem}{Theorem}
\numberwithin{theorem}{section}

\newtheorem{prop}[theorem]{Proposition}

\newtheorem{cor}[theorem]{Corollary}

\newtheorem{lem}[theorem]{Lemma}
\newtheorem{Definition/Theorem}[theorem]{Definition/Theorem}

\theoremstyle{remark}

\newtheorem{remark}[theorem]{Remark}

\newtheorem{Corollary/Definition}[theorem]{Corollary/Definition}

\title{Some sums over irreducible polynomials}
\author{David E Speyer}
\date{}

\begin{document}

\begin{abstract}
We prove a number of conjectures due to Dinesh Thakur concerning sums of the form $\sum_P h(P)$ where the sum is over monic irreducible polynomials $P$ in $\mathbb{F}_q[T]$, the function $h$ is a rational function and the sum is considered in the $T^{-1}$-adic topology. As an example  of our results, in $\FF_2[T]$, the sum $\sum_P \tfrac{1}{P^k - 1}$ always converges to a rational function, and is $0$ for $k=1$.
\end{abstract}

\maketitle{}

\section{Introduction}

Our goal is to explain some identities experimentally discovered by Dinesh Thakur, involving sums over irreducible polynomials in finite fields. We begin by stating the simplest of these identities: Let $\cP$ be the set of irreducible polynomials in $\FF_2[T]$. Then
\[ \sum_{P \in \cP} \frac{1}{P-1}  = 0 . \]
Here the sum must be interpreted as a sum of power series in $T^{-1}$. For example, the first five summands are 
\[\begin{array}{lcc@{}c@{}c@{}c@{}c@{}l}
\frac{1}{T-1} &=& T^{-1} &+& T^{-2} &+& T^{-3} &+ \cdots \\
\frac{1}{(T+1)-1} &=& T^{-1} & & & & & \\
\frac{1}{(T^2+T+1)-1} &=&  & & T^{-2} &+& T^{-3} &+ \cdots \\
\frac{1}{(T^3+T+1)-1} &=&  & &  && T^{-3} &+ \cdots \\
\frac{1}{(T^3+T^2+1)-1} &=&  & &  && T^{-3} &+ \cdots . \\
\end{array}\]
As the reader can see, only finitely many terms contribute to the coefficient of each power of $T^{-1}$, and the coefficient of $T^{-j}$ is $0$ for each $j$.

We now introduce the notation necessary to state our general results. To aid the reader's comprehension, we adopt the following conventions: Integers will always be denoted by lower case Roman letters ($k$, $p$, $q$ \dots); polynomials over finite fields will always be denoted by capital Roman letters ($A$, $F$, $P$ \dots), sets of such polynomials will always be denoted by calligraphic letters ($\cA$, $\cP$, $\cR$, \dots), symmetric polynomials will be denoted by bold letters ($\be_k$, $\bp_k$, \dots).
Of course, there will be other sorts of mathematical objects as well, which we trust the reader to accommodate as they occur.

Let $p$ be a prime and $q$ a power of $p$. Let $\FF_q$ be the field with $q$ elements. Let $\cR$ be the polynomial ring $\FF_q[T]$.
Let $\cK$ be the fraction field $\FF_q(T)$ and let $\wcK$ be the $T^{-1}$-adic completion of $\cK$.
All infinite sums will be understood in the $T^{-1}$-adic topology.

Let $\cP$ be the set of irreducible polynomials in $\cR$; let $\cP_1$ be the set of monic irreducible polynomials.
Here is our main result for the case $p=2$.

\begin{theorem} \label{main2}
If $p=2$ then, for any positive integer $k \equiv 0 \bmod q-1$, the sum
\[ \sum_{P \in \cP_1} \frac{1}{P^{k}-1} \]
is in $\cK$. 
\end{theorem}

The reader may wonder what happens is we sum over all irreducible polynomials rather than monic ones; that is an easy corollary:
\begin{cor} \label{notMonic}
For any positive integer $k$, the sum
\[ \sum_{P \in \cP} \frac{1}{P^{k}-1} \]
is in $\cK$. 
\end{cor}

\begin{proof}
We rewrite the sum as $\sum_{P \in \cP_1} \sum_{a \in \FF_q^{\times}} \tfrac{1}{(aP)^k-1}$.
The corollary then follows from the identity
\[ \sum_{a \in \FF_q^{\times}} \frac{1}{(aX)^k-1} = \frac{1}{X^{\LCM(k, q-1)}-1} \]
in $\FF_q(U)$. To prove this identity, write 
\[ \frac{1}{(aX)^k-1}  = \sum_{j=1}^{\infty} 1/(aX)^{kj} \]
and recall that 
\[ \sum_{a \in \FF_q^{\times}} a^m = \begin{cases} 1 & m \equiv 0 \bmod q-1 \\ 0 & \mbox{otherwise} \end{cases} . \]
\end{proof}

We now discuss the case of a general prime.
Define the rational function $G_p(U)$ by
\[ G_p(U) = \frac{(1-U^p)- (1-U)^p}{p (1-U)^p}. \]
When $p=2$, we have $G_2(U) = \tfrac{2U - 2 U^2}{2 (1-U)^2} = U/(1-U)$, so $G_2(1/P) = 1/(P-1)$.
When $p$ is odd, we have the following alternate expressions for $G_p$:

\begin{prop} \label{GAltForms}
If $p$ is odd, then, as rational functions in $\FF_p(U)$, we have
\[ G_p(U) = \frac{\sum_{j=1}^{p-1} \frac{U^j}{j}}{(1-U)^p} = \sum_{\substack{0 \leq j < \infty \\ j \not \equiv 0 \bmod p}} \tfrac{U^j}{j} . \]
\end{prop}

\begin{proof}
If $p$ is odd, then $(1-U^p)- (1-U)^p = \sum_{j=1}^{p-1} (-1)^{j-1} \binom{p}{j} U^j$. We have
\[ \frac{(-1)^{j-1}}{p} \binom{p}{j} = \frac{(-1)^{j-1} (p-1) (p-2) \cdots (p-j+1)}{1 \cdots 2 \cdots (j-1) j} \equiv \frac{1}{j} \bmod p. \]
This proves the first equality, and the second is immediate.
\end{proof}

\begin{theorem} \label{mainOdd}
For any positive integer $k \equiv 0 \bmod q-1$, the sum
\[ \sum_{P \in \cP_1} G_p(1/P^k) \]
is in $\cK$.
\end{theorem}
As we noted, $G_2(1/X)=1/(X-1)$,  so Theorem~\ref{mainOdd} implies Theorem~\ref{main2}.

\begin{remark}
When $p=2$, we do \emph{not} have $G_2(U) = \sum_{ j \not \equiv 0 \bmod p} \tfrac{U^j}{j}$; the latter sum is $H(U) := \frac{U}{1-U^2}$. However, it is true that $\sum_{P \in \cP_1} H(1/P^k)$ is in $\cK$, because $H(U) = G(U) - G(U^2)$.
\end{remark}

Once again, we have a trivial variant where we sum over $\cP$:
\begin{cor}
For any positive integer $k$, the sum 
\[ \sum_{P \in \cP} G_p(1/P^k) \]
is in $\cK$. 
\end{cor}

\begin{proof}
If $p=2$, we proved this in Corollary~\ref{notMonic}, so we may (and do) assume $p$ is odd.
As in the proof of Corollary~\ref{notMonic}, we rewrite the sum as $\sum_{P \in \cP_1} \sum_{a \in \FF_q^{\times}} G_p(1/(a P)^k)$.
We now need the identity
\[ \sum_{a \in \FF_q^{\times}} G_p((a U)^k) =   \GCD(q-1,k)  G_p(U^{\LCM(q-1,k)})   \]
in $\FF_q(U)$.  To prove this identity, we use the formula $G_p(U) = \sum_{j \not \equiv 0 \bmod p} \tfrac{U^j}{j}$ and the identity
\[ \sum_{a \in \FF_q^{\times}} a^m = \begin{cases} q-1 & m \equiv 0 \bmod q-1 \\ 0 & \mbox{otherwise} \end{cases} . \]
So
\[ \sum_{a \in \FF_q^{\times}} G_p(1/(aU)^k) = \sum_{j \not \equiv 0 \bmod p} \sum_{a \in \FF_q^{\times}} \frac{1}{j (aU)^{kj}} = (q-1)  \sum_{\substack{j \not \equiv 0 \bmod p \\ kj \equiv 0 \bmod q-1}} \frac{1}{j U^{kj}}. \]
Putting $kj = \LCM(q-1,k) \ell$, this is
\[ \begin{array}{ll}
& (q-1) \sum_{\ell \not \equiv 0 \bmod p} \frac{k}{\LCM(q-1,k) \ell \ U^{\LCM(q-1,k) \ell}}  \\
= & \frac{k(q-1)}{\LCM(q-1,k)} G_p(U^{\LCM(q-1,k)})   \\
= &  \GCD(q-1,k)  G_p(U^{\LCM(q-1,k)})   \\
 \end{array}  \]
as required. 
\end{proof}

We also compute explicit values for the sum when $k$ is not too large. 
\begin{theorem} \label{explicit}
Let $k = (q-1) \ell$. If $1 \leq \ell \leq q/p$, then $\sum_{P \in \cP_1} G_p(1/P^k) =0$. If  $q/p+1 < \ell \leq 2 q/p$, then 
\[ \sum_{P \in \cP_1} G_p(1/P^k) =  \ell  \frac{(T^q-T)^{q+1}}{(T^{q^2}-T^q)(T^{q^2}-T)} . \]
\end{theorem}
In principle, our methods are capable of computing $\sum_{P \in \cP_1} G_p(1/P^k)$ for any $k \equiv 0 \bmod q-1$, but they become impractical beyond $\ell = 2 q/p$.

\subsection{History of the problem} Dinesh Thakur suspected such relations should exist, based on heuristics concerning $\zeta$ deformation. He experimentally discovered most of the relations described above in characteristic two, and suspected there should be similar results in odd chracteristic. 
Thakur published these computations in a preprint entitled ``Surprising symmetries in distribution of prime polynomials"~\cite{Thakur}.
At Thakur's suggestion, Terence Tao promoted the problem in posts on his blog and on the Polymath blog~\cite{polymath}. 
I am grateful to Thakur for finding such an elegant problem and to Tao for bringing it to my attention.
My thanks also to all who participated in the discussion on the Polymath blog: Noam Elkies, Ian Finn, Ofir Gorodetsky, Jesse,  Gil Kalai, David Lowry-Duda, Dustin G. Mixon, John Nicol, Partha Solapurkar, John Voight, Victor Wang, Qiaochu Yuan, Joshua Zelinsky.
 The author is supported by NSF grant DMS-1600223.

\section{The Carlitz exponential, and symmetric polynomials}

The main tool in our proofs is the theory of the Carlitz exponential.
Put 
\[ D_i = (T^{q^i}-T)(T^{q^i}-T^q)(T^{q^i}-T^{q^2}) \cdots (T^{q^i} - T^{q^{i-1}}) . \]
Define
\[
 e_C(Z) = \sum_{j=0}^{\infty} \frac{T^{q^j}}{D_j}
\]
this sum is $T^{-1}$-adically convergent for any $Z \in \wcK$. We will make use of the product identity:
\[ \frac{e_C(\overline{\pi} Z)}{\overline{\pi} Z} = \prod_{A \in \cR \setminus \{ 0 \}} \left( 1+\frac{Z}{A} \right) \]
where $\overline{\pi} \in \wcK(\! \sqrt[q-1]{-T})$ is given by
\[ \overline{\pi} = \frac{T \sqrt[q-1]{-T}}{\prod_{A \in \cR \setminus \{ 0 \}} (1-(TA)^{-1})} . \]
See, for example,~\cite[Theorem 3.2.8]{Goss}.
This identity should be thought of as  similar to Euler's identity:
\[ \frac{\sin (\pi z)}{\pi z} = \prod_{a \in \ZZ \setminus \{ 0 \}} \left( 1+\frac{z}{a} \right) . \]

We introduce the notations $\cA$ for the nonzero polynomials of $\cR$, and $\cA_1$ for the monic polynomials.

Writing $\be_k$ for the elementary symmetric function of degree $k$, this implies
\[ \be_k(1/A)_{A \in \cA} = 
\begin{cases} \overline{\pi}^k/D_{j} & k = q^j-1\\ 0 & \mbox{otherwise} \end{cases} \]
Since the ring of symmetric polynomials is generated by the $\be_k$, we deduce
\begin{prop} \label{Key}
If $\bf$ is a homogenous symmetric polynomial of degree $k$, then $\bf(1/A)_{A \in \cA}$ is in $\overline{\pi}^k \cK$.
\end{prop}
Here we note that $\bf(1/A)_{A \in \cA}$ is always defined, since only finitely many terms contribute to the coefficient of any particular power of $T^{-1}$.

The above considers symmetric polynomials in $\{ 1/A \}_{A \in \cA}$, but we would rather restrict to the case of $A$ monic. To this end, we have
\begin{prop} \label{monicElementary}
\[ \be_{\ell}(1/A^{q-1})_{A \in \cA_1} = 
\begin{cases} (-1)^{\ell} \overline{\pi}^{\ell(q-1)}/D_j & \ell = \tfrac{q^j-1}{q-1} \\ 0 & \mbox{otherwise}
\end{cases}\]
\end{prop}

\begin{proof}
Grouping together scalar multiples of the same polynomial in the Carlitz product identity, we have
\[ \frac{e_C(\overline{\pi} Z)}{\overline{\pi} Z} = \prod_{A \in \cA_1} \left( 1-\frac{Z^{q-1}}{A^{q-1}} \right). \]
Equate coefficients of $Z^{\ell (q-1)}$ on both sides.
\end{proof}

%

\begin{cor} \label{monicKey}
If $\bf$ is a homogenous symmetric polynomial of degree $\ell$, then $\bf(1/A^{q-1})_{A \in \cA_1}$ is in $\overline{\pi}^{\ell (q-1)} \cK$.
\end{cor}

\section{Proofs of rationality}

We now have enough background to prove  Theorem~\ref{mainOdd} and, hence, Theorem~\ref{main2}. 
Throughout, let $k \equiv 0 \bmod q-1$.

Consider the symmetric polynomial 
\[ \bg_p(X_1, X_2, \ldots,) := \frac{1}{p} \left( \left( \sum X_i \right)^p - \sum X_i^p \right). \]
The polynomial $\bg_p$ has integer coefficients, so we may discuss plugging elements of $\cK$ into it.


Let $C$ be the cyclic group of order $p$, and let $C$ act on $\cA_1^p$ by rotating coordinates. Let $\Delta$ denote the small diagonal: $\Delta := \{ (A,A,\ldots,A) \}  \subset \cA_1^p$.
Then
\[ \bg_p(1/A^k)_{A \in \cA_1} = \sum_{(A_1, \ldots, A_p) \in (\cA_1^p \setminus \Delta)/C } \frac{1}{A_1^k A_2^k \cdots A_p^k}. \]
The sum is over cosets for the free action of $C$ on $\cA^p \setminus \Delta$.

Let
\[ \Phi = \{ (A_1, \ldots, A_p) \in \cA_1^p : \GCD(A_1, \ldots, A_p) = 1 \}. \]
Any $(A_1, \ldots, A_p) \in \cA_1^p$ can be uniquely factored as $A_i = D B_i$ for some $D \in \cA_1$ and $(B_1,\ldots, B_p) \in \Phi$. 
So we can factor the above sum as 
\[ \bg_p(1/A^k)_{A \in \cA_1} = \left( \sum_{D \in \cA_1} \frac{1}{D^{kp}} \right) \left( \sum_{(B_1, \ldots, B_p) \in (\Phi \setminus \{(1,\ldots,1)\})/C } \frac{1}{B_1^k B_2^k \cdots B_p^k} \right). \]

Now, from Proposition~\ref{monicKey}, $\bg_p(1/A^k)_{A \in \cA}$, is in $\overline{\pi}^{pk} \cK$. 
Also from Proposition~\ref{monicKey}, $\sum_{D \in \cA_1} 1/D^{kp}$ is in $\overline{\pi}^{pk} \cK$, and a quick computation shows that this sum is $1$ plus terms in $T^{-1} \FF_q[[T^{-1}]]$, so it is not zero.
We deduce that
\[  \sum_{(B_1, \ldots, B_p) \in (\Phi \setminus (1,\ldots,1))/C } \frac{1}{B_1^k B_2^k \cdots B_p^k} \in \cK . \]

For $B \in \cA_1$, let $\Psi(B)$ be the set of $p$-tuples $(B_1, B_2, \ldots, B_p)$ for which  $\GCD(B_1, \ldots, B_p)=1$ and $\prod B_i = B$. 
Let $\psi(B) = \# \Psi(B)$. So we have shown that
\[ \sum_{B \in \cA_1 \setminus \{ 1 \}} \frac{\psi(B)/p}{B^k}  \in \cK. \]
Here, to interpret the numerator, we must divide $\psi(B)$ by $p$ as integers and only then consider the quotient in $\FF_p$.

If $B = P_1^{k_1} P_2^{k_2} \cdots P_r^{k_r}$ then there is an easy bijection between $\Psi(B)$ and $\Psi(P_1^{k_1}) \times \Psi(P_2^{k_2}) \times \cdots \times \Psi(P_r^{k_r})$, so $\psi(B) = \prod \psi(P_i^{k_i})$. 
 If $P$ is irreducible then $\psi(P^r)$ is divisible by $p$ for any $r>0$, since $C$ acts freely on $\Psi(P^r)$. So, if $B$ is divisible by two different irreducible polynomials, then $\psi(B)$ is divisible by $p^2$. So we can rewrite the sum as
\[ \sum_{P \in \cP_1} \sum_{r=1}^{\infty} \frac{\psi(P^r)/p}{P^{rk}} .\]

We now compute $\psi(P^r)$; which is the number of $p$-tuples $(P^{r_1},  \ldots, P^{r_p})$ with $\prod P^{r_i} = P^r$ and $\GCD(P^{r_1}, \ldots, P^{r_p})=1$. In other words, we must count $(r_1, \ldots, r_p) \in \ZZ_{\geq 0}^p$ with $\sum r_i = r$ and $\min(r_1, \ldots, r_p)=0$. 
The number of $(r_1, \ldots, r_p) \in \ZZ_{\geq 0}^p$ with $\sum r_i = r$ is the coefficient of $U^r$ in $1/(1-U)^p$. 
In order to impose $\min(r_1, \ldots, r_p)=0$, we subtract off the terms with $\min(r_1, \ldots, r_p) > 0$. These are in bijection with $(s_1, \ldots, s_p)  \in \ZZ_{\geq 0}^p$ with $p+\sum s_i = r$. So $\psi(p^r)$ is the coefficient of $U^r$ in $1/(1-U)^p - U^p/(1-U)^p$. In other words, $\sum_{r=0}^{\infty} \psi(P^r) U^r = (1-U^p)/(1-U)^p$.
So 
\[ \sum_{r=1}^{\infty} \tfrac{\psi(P^r)}{p} U^r = \frac{1}{p} \left( \frac{1-U^p}{(1-U)^p} -1 \right) = G_p(U).\]
 We deduce that $\sum_{r=1}^{\infty} \frac{\psi(P^r)/p}{P^{rk}} = G_p(1/P^k)$. We have now shown that $\sum_{P \in \cP_1} G_p(1/P^k) \in \cK$, as claimed. \qedsymbol

We record the specific formula we have proved:
\begin{prop} \label{formula}
Let $k$ be a positive integer.  Then
\[ \sum_{P \in \cP_1} G_p(1/P^k) = \frac{\bg_p(1/A^k)_{A \in \cA_1}}{\sum_{A \in \cA_1} 1/A^{pk}} \]
\end{prop}
We will rewrite this formula in various ways in Section~\ref{sec computations}. We remark that this formula is correct even if $k$ is not divisible by $q-1$, although we have only shown the ratio is in $\cK$ when $k \equiv 0 \bmod q-1$. The denominator of this formula is $\zeta(pk) = \zeta(k)^p$ where $\zeta$ is the Goss $\zeta$-function~\cite{Goss2}.

\section{Vanishing}

We will now prove the claim in Theorem~\ref{explicit} that the sum vanishes when $k=(q-1) \ell$ for $1 \leq \ell \leq q/p$. 
From Proposition~\ref{formula}, it is equivalent to show that $\bg_p(1/A^{\ell(q-1)})_{A \in \cA_1}=0$.
To this end, we must explicitly write $\bg_p(1/A^{\ell(q-1)})$ as a polynomial in the $\be_k(1/A^{q-1})$.

The variables $\lambda$ or $\mu$ will always denote partitions, meaning weakly decreasing sequences $(\lambda_1, \lambda_2, \ldots, \lambda_r)$ of positive integers; sums over $\lambda$ or $\mu$ implicitly contain the condition that the summation variable is a partition.

We define $\be_{\lambda} = \prod_s \be_{\lambda_s}$. 
The symmetric polynomials $\be_{\lambda}$ form an integer basis for the symmetric polynomials with integer coefficients.

\begin{lem} \label{OneCoeffVanish}
Write
\[ \bg_p(X_1^{\ell}, X_2^{\ell}, \ldots,) = \sum_{|\lambda|=p\ell} c_{\lambda} \be_{\lambda}(X_1, X_2, \ldots) \]
for some integers $c_{\lambda}$. Then $c_{11\cdots 1}=0$. 
\end{lem}

\begin{proof}
Note that $\be_{11\cdots 1}$ is the only $\be_{\lambda}$ with a nonzero coefficient of $X_1^{p \ell}$. The coefficient of $X_1^{p\ell}$ in $\bg_p(X_1^{\ell}, X_2^{\ell}, \ldots)$ is clearly $0$.
\end{proof}

Now, suppose that $\ell\leq q/p$, so we have $p \ell < q+1$. So any partition $(\lambda_1, \ldots, \lambda_r)$ of $p \ell$ other than $(1,1,\ldots, 1)$ contains a $\lambda_i$ between $2$ and $q$.
By Lemma~\ref{monicElementary}, $\be_{m}(1/A^{q-1})_{A \in \cA_1}=0$ for $2 \leq m \leq q$, so $\be_{\lambda}(1/A^{q-1})_{A \in \cA_1}=0$ whenever $\lambda$ is a partition of $p \ell$ other than $(1,1,\ldots,1)$. We deduce that $\bg_p(1/A^{q-1})_{A \in \cA_1}=0$ as desired. \qedsymbol

\section{Computations for small $k$} \label{sec computations}

In this section, we will discuss the computation of $\sum_{P \in \cP_1} G_p(1/P^k)$ for $k \equiv 0 \bmod q-1$ and, in particular, prove the remaining half of Theorem~\ref{explicit}. Our strategy is to combine Propositions~\ref{formula} and~\ref{monicElementary}.
We must compute $\bg_p(1/A^{\ell(q-1)})_{A \in \cA_1}$ and $\sum_{A \in \cA_1} 1/A^{pk}$. Note the latter is $\left( \bp_{\ell}(1/A^{q-1})_{A \in \cA_1} \right)^p$, where $\bp_d(X_1, X_2, \ldots)$ is the power sum symmetric function $\sum X_i^d$.
We write $k = (q-1) \ell$.


Put
\[ \begin{array}{rcl}
\bg_p(X_1^{\ell}, X_2^{\ell}, \dots) &=& \sum_{|\lambda| = \ell p} c_{\lambda} \be_{\lambda}(X_1,X_2, \ldots) \\
\bp_{\ell}(X_1, X_2, \dots) &=& \sum_{|\mu| = \ell} d_{\mu} \be_{\mu}(X_1,X_2, \ldots) . \\
\end{array} \]
Note that $\be_m(1/A^{q-1})_{A \in \cA_1}=0$ unless $m$ is of the form $(q^j-1)/(q-1)$. 
So we only need to sum over partitions where all the parts of $\lambda$ are of the form $(q^j-1)/(q-1)$.

\textbf{From now on, we now impose that $q/p+1 \leq \ell \leq 2q/p$.}
So $\ell < q+1$. Any partition of $\ell$ cannot contain any parts of size $(q^j-1)/(q-1)$, for $j > 1$. Similarly, 
$p \ell  < 2q+2$, so a partition of $p \ell$ can contain at most one part of size $(q^2-1)/(q-1) = q+1$ and no parts of size $(q^j-1)/(q-1)$ for $j > 2$.
We deduce that the only terms which contribute to our final answer come from $\lambda = (1,1,\ldots,1)$ or $\lambda = (q+1,1,1,\ldots, 1)$ when computing $\bg_p(1/A^{\ell(q-1)})_{A \in \cA_1}$, and from $\mu = (1,1,\ldots,1)$ in computing $\left( \bp_{\ell}(1/A^{q-1})_{A \in \cA_1} \right)^p$.
Moreover, from Lemma~\ref{OneCoeffVanish}, the coefficient $c_{(1,1,\ldots, 1)}$ is zero. 

We deduce that
\[ \begin{array}{rcl}
\sum_{P \in \cP_1} G_p(1/P^k) &=& \frac{c_{(q+1, 1^{p\ell-q-1})} \ \be_{(q+1, 1^{p\ell-q-1})}(1/A^{q-1})_{A \in \cA_1}}{\left( d_{1^{\ell}}\ \be_{1^{\ell}}(1/A^{q-1})_{A \in \cA_1} \right)^{p} } \\
&=& \frac{c_{(q+1, 1^{p\ell-q-1})} \ \be_{q+1}(1/A^{q-1})_{A \in \cA_1} \left( \be_{1}(1/A^{q-1})_{A \in \cA_1}\right)^{p\ell-q-1} }{ d_{1^{\ell}}\  \left( \be_{1}(1/A^{q-1})_{A \in \cA_1} \right)^{p\ell} }\\
&=& \frac{c_{(q+1, 1^{p\ell-q-1})} \ \be_{q+1}(1/A^{q-1})_{A \in \cA_1}}{ d_{1^{\ell}}\  \left( \be_{1}(1/A^{q-1})_{A \in \cA_1} \right)^{q+1} }\\
\end{array}
. \]
Here $1^r$ is shorthand for $r$ parts equal to $1$.

We now use Proposition~\ref{monicElementary}. The powers of $\overline{\pi}$ and $(-1)$ cancel to give
\[
\sum_{P \in \cP_1} G_p(1/P^k) = \frac{c_{(q+1, 1^{p\ell - q -1})}}{d_{1^{\ell}}}  \frac{D_1^{q+1}}{D_2} =  \frac{c_{(q+1, 1^{p\ell - q -1})}}{d_{1^{\ell}}} \frac{(T^q-T)^{q+1}}{(T^{q^2}-T^q)(T^{q^2}-T)}.
\]

To finish the computation, we must find $c_{q+1, 1^{ps-1}}$ and $d_{1^{\ell}}$. The latter is easy: Comparing coefficients of $X_1^{\ell}$ on both sides of $\bp_{\ell}(X_1, X_2, \dots) = \sum_{|\mu| = \ell} d_{\mu} \be_{\mu}(X_1,X_2, \ldots)$, we deduce that $d_{1^{\ell}}=1$. 

To compute $c_{(q+1, 1^{p\ell-q-1})}$, we begin with the formula 
\[ \bg_p(X_1^{\ell}, X_2^{\ell}, \dots)  = \frac{1}{p} \left(\bp_{\ell}(X_1, X_2, \ldots,)^p - \bp_{p \ell}(X_1, X_2, \ldots) \right). \]
For brevity, we write $\bf(X)$ to indicate that the inputs to a symmetric polynomial are $(X_1, X_2, \ldots)$. 
Note that we are working with symmetric polynomials with integer coefficients, so it makes sense to divide by $p$.

We rewrite the right hand side of the previous equation as
\[ \frac{1}{p} \left( \left( \be_1(X)^{\ell}+\! {\scriptstyle \cdots} \right)^p - \left( \be_1(X)^{p \ell} + d_{q+1, 1^{p\ell-q-1}} \be_{q+1}(X) \be_1(X)^{p\ell-q-1} +\! {\scriptstyle \cdots} \right) \right). \]
Here the ellipses denote terms $\be_{\lambda}$ where $\lambda$ has some part not of the form $(q^j-1)/(q-1)$. 
We deduce that
\[ c_{q+1, 1^{p\ell-q-1}} = - \frac{1}{p} d_{q+1, 1^{p\ell-q-1}} . \]

Now, observe the identity
\begin{multline*} 
\sum_j \frac{(-1)^{j-1} \bp_j(X) U^j}{j} = \sum_i \log(1+X_i U) \\ = \log \prod_i (1+X_i U) = \log \left( 1+ \sum_{m=1}^{\infty} \be_m(X) U^m \right) . \end{multline*}
The coefficient of $U^{p \ell}$ on the left is $\tfrac{(-1)^{p \ell}}{p \ell} \bp_{p \ell}$. Expanding the $\log$ on the right hand side as a Taylor series, only one term contributes to $U^{p \ell} \be_{q+1} \be_1^{p\ell -q-1}$. So we obtain
\[ \frac{(-1)^{p \ell-1}}{p \ell} \bp_{p \ell}(X) =  \frac{(-1)^{p\ell-q-1}}{p\ell-q} \binom{p\ell-q}{1} \be_{q+1}(X) \be_1^{p\ell-q-1}(X) + \cdots \]
where the ellipses denote a sum of $\be_{\lambda}$ other than $\be_{q+1}(X) \be_1^{p\ell-q-1}(X) $.
So $d_{q+1, 1^{p\ell-q-1}} = (-1)^{q} p \ell$ and $c_{q+1, 1^{ps-1}} = (-1)^{q-1} \ell$.
Plugging into our previous formula, and using that $(-1)^{q-1} \equiv 1 \bmod p$,
\[
\sum_{P \in \cP_1} G_p(1/P^k) = \ell  \frac{(T^q-T)^{q+1}}{(T^{q^2}-T^q)(T^{q^2}-T)}.
 \]
This concludes the proof of Theorem~\ref{explicit}. \qedsymbol

We conclude by verifying one Thakur's conjectures which goes beyond the range $\ell \leq 2 q/p$. 
Let $p=q=2$. Thakur conjectures
\[ \sum_{P \in \cP_1} \frac{1}{P^3-1} = \frac{1}{T^4+T^2} . \]

We begin by computing
\[\begin{array}{rcl} \bp_3(X) &=& \be_1(X)^3 + 3 \be_3(X) - 3 \be_2(X) \be_1(X) \\
 \bp_3(X)^2 &=& \be_1(X)^6 + 6 \be_1(X)^3 \be_3(X) + 9 \be_3(X)^2 + \cdots . \\ \end{array}\]
Here and in the following equations, the ellipses denote $\be_{\lambda}$ terms where $\lambda$ contains a part other than $1$ and $3$. (Note that $(2^3-1)/(2-1) = 7$, too large to contribute to a symmetric polynomial of degree $6$.)
Similarly, 
\[ \bp_6(X) = \be_1(X)^6 + 6 \be_1(X)^3 \be_3(X) + 3 \be_3(X)^2 + \cdots .\]
So
\[ \bg_2(X_1^3, X_2^3, \ldots) = \frac{1}{2} \left( \bp_3(X)^2 - \bp_6(X) \right) = 3 \be_3(X)^2 + \cdots  \]
and (recall that we are working modulo $2$)
\[ \bg_2(1/A^3)_{A \in \cA_1} = \left( \be_3(1/A)_{A \in \cA_1} \right)^2 = \frac{\overline{\pi}^6}{D_2^2} = \frac{\overline{\pi}^6}{(T^4-T^2)^2(T^4-T)^2} . \]

Similarly, 
\[ \begin{array}{rcl}
\bp_6(1/A)_{A \in \cA_1} &=& \left( \be_1(1/A)_{A \in \cA_1} \right)^6 + \left( \be_3(1/A)_{A \in \cA_1} \right)^2 \\
&=& \left( \frac{\overline{\pi}}{D_1} \right)^6 + \left( \frac{\overline{\pi}^3}{D_2} \right)^2 \\
&=& \overline{\pi}^6 \left( \left( \frac{1}{T^2-T} \right)^6 + \left( \frac{1}{(T^4-T^2) (T^4-T)} \right)^2 \right)
\end{array} \]

We verify Thakur's claim:
\[  \sum_{P \in \cP_1} \frac{1}{P^3-1}  = 
\frac{1/( (T^4-T^2)^2(T^4-T)^2 )}{1/(T^2-T)^6 + 1/( (T^4-T^2)^2(T^4-T)^2 )} = \frac{1}{T^4+T^2} .\]

\end{document}